\documentclass[reqno,12pt]{amsart}

\usepackage[latin1]{inputenc}

\usepackage{enumerate}
\usepackage{amstext,amsmath,amsthm,amsfonts,amssymb,amscd}
\usepackage{latexsym,mathrsfs,eufrak,dsfont,euscript}
 \usepackage{ulem,cancel}
\usepackage{fullpage}
\usepackage{txfonts}

\usepackage{color}
\definecolor{rojo}{RGB}{221,0,0}

\usepackage{hyperref} 

\usepackage{tikz}
\usetikzlibrary{arrows,calc}
\tikzset{
>=stealth',
help lines/.style={dashed, thick},
axis/.style={<->},
important line/.style={thick},
connection/.style={thick, dotted},
}

\usetikzlibrary{matrix}
\usepackage{pstricks}
\usepackage{pst-plot}
\usepackage{pst-all}
\usepackage{graphicx}

\newtheorem{theorem}{Theorem}

\newtheorem{lemma}[theorem]{Lemma}
\newtheorem{corollary}[theorem]{Corollary}
\theoremstyle{definition}

\theoremstyle{remark}

\numberwithin{equation}{section}

\DeclareMathOperator{\supp}{supp}
\newcommand{\abs}[1]{\left\vert#1\right\vert}
\newcommand{\set}[1]{\left\{#1\right\}}
\newcommand{\proin}[2]{\left<#1,#2\right>}
\newcommand{\norm}[1]{\left\Vert#1\right\Vert}

\allowdisplaybreaks
\begin{document}
\title[]{The dyadic fractional diffusion kernel as a central limit}
%
%


\author[]{Hugo Aimar}
\email{haimar@santafe-conicet.gov.ar}
\author[]{Ivana G\'{o}mez}
\email{ivanagomez@santafe-conicet.gov.ar}
\author[]{Federico Morana}
\email{fmorana@santafe-conicet.gov.ar}
\thanks{The research was supported  by CONICET, ANPCyT (MINCyT) and UNL}

\begin{abstract}
In this paper we obtain the fundamental solution kernel of dyadic diffusions in $\mathbb{R}^+$ as a Central Limit of dyadic mollification of iterations of stable Markov kernels. The main tool is provided by the substitution of classical Fourier analysis by Haar wavelet analysis.
\end{abstract}

\maketitle

\section{Introduction}
The analysis of solutions of nonlocal problems in PDE, has received new impulse after the remarkable results obtained by Caffarelli and Silvestre \cite{CaSi07}. For a probabilistic view of this problems see \cite{Val09}, \cite{Valdinocibook16}. Recently in \cite{AcAimFCAA},\cite{AcAimCzech},\cite{AiBoGo13}, a dyadic version of the fractional derivative was introduced and an associated diffusion was solved.

The classical diffusion process, described by  the heat equation $\tfrac{\partial u}{\partial t}=\Delta u$, where  $\Delta$
denotes the space Laplacian, has as a fundamental solution  the Weierstrass kernel $W_t(x)= (4\pi
t)^{-d/2}e^{-\abs{x}^2/4t}$, which is the central limit distribution, for $n\to\infty$, of $\sqrt{n}^{-1}\sum_{j=1}^{n}X_j$,
where the $X_j$'s are identically distributed independent random variables with finite variance $t$ and vanishing mean value.
For our later analysis it is convenient to write the convergence in distribution of $n^{-1/2}\sum_{j=1}^n X_j$ to $W_t$ in terms of the common distribution of the random variables $X_j$, $j\in \mathbb{N}$. For the sake of simplicity let us assume that this distribution is given by the density $g$ in $\mathbb{R}^d$. In other words, $\mathscr{P}(\{X_j\in B\})=\int_Bg(x)dx$ where $B$ is a Borel set in $\mathbb{R}^d$. Hence since the random variables $X_j$ are independent the distribution of $S_n=\sum_{j=1}^nX_j$ is given by the convolution $g^n$ of $g$ $n$-times. Precisely, with $g^n=g\ast\cdots\ast g$ $n$-times, we have that $\mathscr{P}(\{S_n\in B\})=\int_Bg^n(x)dx$. On the other hand, $\mathscr{P}(\{n^{-1/2}\sum_{j=1}^nX_j\in B\})=\mathscr{P}(\{S_n\in\sqrt{n}B\})=\int_B(g^n)_{\sqrt{n}}(x)dx$, with $(g^n)_{\sqrt{n}}$ the mollification of $g^n$ by $\sqrt{n}$ in $\mathbb{R}^d$. Precisely, $(g^n)_{\sqrt{n}}(x)=n^{-d/2}g^n(\sqrt{n}x)$. These observations allows to read the CLT as a vague or Schwartz weak convergence of $(g^n)_{\sqrt{n}}(x)$ to $W_t(x)$ when $n\to\infty$. For every $f$ continuous and compactly supported in $\mathbb{R}^d$, we have that $\int_{\mathbb{R}^d}(g^n)_{\sqrt{n}}(x)f(x)\to\int_{\mathbb{R}^d}W_t(x)f(x) dx$ as $n\to\infty$. Since we shall be working in a non-translation invariant setting, to get the complete analogy we still rewrite the CLT as the weak convergence of the sequence of Markov kernel $K^n_{\sqrt{n}}(x,y)=(g^n)_{\sqrt{n}}(x-y)$ to the  Markov Weierstrasss kernel $W_t(x-y)$. The kernel $K^n_{\sqrt{n}}(x,y)=\idotsint_{\mathbb{R}^{d-1}} g_{\sqrt{n}}(x-x_1)g_{\sqrt{n}}(x_1-x_2)\cdots g_{\sqrt{n}}(x_{n-1}-y)dx_1dx_2\cdots dx_{n-1}$ corresponds to the kernel of the $n$-th iteration of the operator $T_{\sqrt{n}}f(x)=\int_{\mathbb{R}^d}g_{\sqrt{n}}(x-y)f(y) dy$. The difference in the rhythms of the upper index $n$ of the iteration and the  lower index $\sqrt{n}$ of mollification is related to the property of finite variance of $g$. In the problems considered here the Markov kernels involved have heavy tails and the central equilibria takes place for different proportions between iteration and mollification. There are many books where the classical CLT and some of its extensions are masterly exposed. Let us refer to \cite{Chungbook} as one of them.

In this paper we shall be concerned with diffusions of fractional type associated with dyadic differentiation in the space. The basic setting for our diffusions is $\mathbb{R}^+=\{x\in \mathbb{R}: x>0\}$. In \cite{AcAimCzech} it is proved that the function $u(x,t)$ defined for $x\in \mathbb{R}^+$ and $t>0$, given by
\begin{equation*}
u(x,t)=\sum_{h\in\mathscr{H}}e^{-t\abs{I(h)}^{-s}}\proin{u_0}{h}h(x),
\end{equation*}
with $\mathscr{H}$ the standard Haar system in $L^2(\mathbb{R}^+)$, $I(h)$ the support of $h$ and
$\proin{u_0}{h}=\int_{\mathbb{R}^+}u_0(x)h(x) dx$, solves the problem
\begin{equation*}
\left
\{\begin{array}{ll}
\frac{\partial u}{\partial t}=D^{s} u,\, & x\in\mathbb{R}^{+}, t>0;\\

u(x,0)=u_0(x),\,  & x\in \mathbb{R}^+.
\end{array}
\right.
\end{equation*}
with
\begin{equation}\label{eq:derivativefractionalDs}
D^{s}g(x)=\int_{y\in \mathbb{R}^+}\frac{g(x)-g(y)}{\delta (x,y)^{1+s}} dy
\end{equation}
for $0<s<1$ and $\delta(x,y)$ the dyadic distance in $\mathbb{R}^{+}$ (see Section~\ref{sec:dyadycAnalysis} for definitions). The main point in the prove of the above statement is provided by the spectral analysis for $D^s$ in terms of Haar functions. In fact, $D^s h=\abs{I(h)}^{-s}h$. When $0<s<1$, since $h$ is a Lipschitz function with respect to $\delta$, the integral in \eqref{eq:derivativefractionalDs} defining $D^sh$ is absolutely convergent. For the case $s=1$ this integral is generally not convergent, nevertheless the operator $D^1$ is still well defined on the Sobolev type space of those function in $L^2(\mathbb{R}^+)$ such that the Haar coefficients $\proin{f}{h}$ satisfy the summability condition $\sum_{h\in\mathscr{H}}\tfrac{\abs{\proin{f}{h}}^2}{\abs{I(h)}^2}<\infty$. For those functions $f$ the first order nonlocal derivative is given by $D^1 f=\sum_{h\in\mathscr{H}}\tfrac{\proin{f}{h}}{\abs{I(h)}}h$. Moreover, with $u_0\in L^2(\mathbb{R}^+)$, the function
\begin{equation*}
u(x,t)=\int_{\mathbb{R}^+}K(x,y;t)u_0(y) dy,
\end{equation*}
with
\begin{equation}\label{eq:NucleoHaarDifusiones}
K(x,y;t)=\sum_{h\in\mathscr{H}}e^{-t\abs{I(h)}^{-1}}h(x)h(y),
\end{equation}
solves
\begin{equation*}
(P) \left
\{\begin{array}{ll}
\frac{\partial u}{\partial t}=D^{1} u,\, & x\in\mathbb{R}^{+}, t>0;\\

u(x,0)=u_0(x),\,  & x\in \mathbb{R}^+.
\end{array}
\right.
\end{equation*}
Notice that for each $t>0$ the function of $x\in \mathbb{R}^+$, $u(x,t)$ is in the dyadic Sobolev space and its $D^1$ space derivative belongs to $L^2(\mathbb{R}^+)$.

The kernel $K(\cdot,\cdot;t)$ for fixed $t>0$ is not a convolution kernel. Nevertheless it can be regarded as a Markov transition kernel which, as we shall prove, depends only on $\delta(x,y)$.

In this note we prove that the Markov kernel family $K(\cdot,\cdot;t)$ is the central limit of adequate simultaneous iteration and mollification
of elementary dyadic stable Markov kernels. We shall precisely define stability later, but heuristically it means that the kernel behaves at infinity like a power law of the dyadic distance. The main result is contained in Theorem~\ref{thm:mainresult} in Section~\ref{sec:mainresult}. The basic tool for the proof of our results is the Fourier Haar analysis induced on $\mathbb{R}^+$ by the orthonormal basis of Haar wavelets.

The paper is organized as follow. In Section~\ref{sec:dyadycAnalysis} we introduce the basic facts from dyadic analysis on $\mathbb{R}^+$, in particular the Haar system as an orthonormal basis for $L^2(\mathbb{R}^+)$ and as an unconditional basis for $L^p(\mathbb{R}^+)$, $1<p<\infty$. Section~\ref{sec:Markovdyadickernels} is devoted to introduce the Markov type dyadic kernels. The spectral analysis of the integral operators generated by Markov type dyadic kernels is considered in \S~\ref{sec:spectralanalysis}. Section~\ref{sec:stability} is devoted to introduce the concept of stability and to prove that the kernel in \eqref{eq:NucleoHaarDifusiones} is $1$-stable with parameter $\tfrac{2}{3}t$. The iteration and mollification operators and their relation with stability are studied in Section~\ref{sec:iterationmollification}. Finally in Section~\ref{sec:mainresult} we state and prove our main result: spectral and $L^p(\mathbb{R}^+)$ ($1<p<\infty$) convergence to the solution of (P).

\section{Some basic dyadic analysis}\label{sec:dyadycAnalysis}
Let $\mathbb{R}^+$ denote the set of nonnegative real numbers. A dyadic interval is a subset of $\mathbb{R}^+$ that can be written as $I=I^j_k=[k2^{-j},(k+1)2^{-j})$ for some integer $j$ and some nonnegative integer $k$. The
family $\mathcal{D}$ of all dyadic intervals can be organized by levels of resolution as follows; $\mathcal{D}=\cup_{j\in \mathbb{Z}}\mathcal{D}^j$, where $\mathcal{D}^j=
\set{I^j_k: k=0,1,2,\ldots}$. The dyadic distance induced on $\mathbb{R}^+$ by $\mathcal{D}$ and the Lebesgue measure is defined by $\delta(x,y)=\inf\set{\abs{I}: I\in
\mathcal{D}, x\in I, y\in I}$ where $\abs{E}$ denotes the one dimensional Lebesgue measure of $E$. It is easy to check that $\delta$ is a distance (ultra-metric) on $\mathbb{R}^+$
 and that, since $\abs{x-y}=\inf\{\abs{J}: x\in J, y\in J, J=[a,b), 0\leq a<b<\infty\}$, $\abs{x-y}\leq\delta(x,y)$. Of course the two distances are not equivalent. Pointwise the function $\delta(x,y)$ is larger than the usual distance $d(x,y)=\abs{x-y}$. Set $B_\delta(x,r)=\{y\in \mathbb{R}^+: \delta(x,y)<r\}$ to denote the $\delta$-ball centered a $x$ with positive radius $r$. Then $B_\delta(x,r)$ is the largest dyadic interval containing $x$ with Lebesgue measure less than $r$. For $r>0$, let $j\in \mathbb{Z}$ be such that $2^j<r\leq 2^{j+1}$. Then, for $x\in \mathbb{R}^+$, $B_\delta(x,r)=I$ with $x\in I\in\mathcal{D}$, $2^j=\abs{I}<r\leq 2^{j+1}$. So that $\tfrac{r}{2}\leq\abs{B_\delta(x,r)}<r$. This normality property of $(\mathbb{R}^+,\delta)$ equipped with Lebesgue measure shows that the $\delta$-Hausdorff dimension of intervals in $\mathbb{R}^+$ is one. In particular the integral singularities that negative powers of $\delta$ and $d$ produce have
the same orders. Precisely, for fixed $x\in \mathbb{R}^+$ the functions of $y\in \mathbb{R}^+$ defined by $\delta^{\alpha}(x,y)$ and $\abs{x-y}^\alpha$ have the same local and global integrability properties for $\alpha\in \mathbb{R}$.

\smallskip
\begin{lemma}\label{lemma:deltaintegrability}
\quad
\begin{enumerate}[(a)]
\item The level sets $L(\lambda)=\{(x,y):\delta(x,y)=\lambda\}$ are empty if $\lambda$ is not an integer power of two. On the other hand
$L(2^j)=\cup_{I\in\mathcal{D}^j}(I_l\times I_r)\cup (I_r\times I_l)$ with $I_l$ and $I_r$, the left and right halves of $I\in\mathcal{D}^j$. Hence, $\delta(x,y)=\sum_{j\in \mathbb{Z}}2^j\chi_{L(2^j)}(x,y)$.

\medskip
\item For $x\in \mathbb{R}^+$ and $r>0$ we have,
\begin{enumerate}[b-i)]

\medskip
\item
$\frac{c(\alpha)}{2^{1+\alpha}}r^{1+\alpha}\leq \int_{y\in B_{\delta}(x,r)} \delta^{\alpha}(x,y) dy \leq c(\alpha)r^{1+\alpha}$
for $\alpha>-1$ with $c(\alpha)=2^{-1}(1-2^{-(1+\alpha)})^{-1}$;

\medskip
\item
$\int_{B_{\delta}(x,r)} \delta^{\alpha}(x,y) dy= +\infty$
for $\alpha\leq -1$;

\medskip
\item
$\tilde{c}(\alpha)r^{1+\alpha}\leq\int_{\{y: \delta(x,y)\geq r\}} \delta^{\alpha}(x,y) dy\leq\frac{\tilde{c}(\alpha)}{2^{1+\alpha}} r^{1+\alpha}$
for $\alpha < -1$ with $\tilde{c}(\alpha)=2^{-1}(1-2^{1+\alpha})^{-1}$;

\medskip
\item
$\int_{\{y: \delta(x,y)\geq r\}} \delta^{\alpha}(x,y) dy= +\infty$
for $\alpha\geq -1$.
\end{enumerate}
\end{enumerate}
\end{lemma}
\begin{proof}[Proof of (a)] Let $j\in \mathbb{Z}$ fixed. Then $\delta(x,y)=2^j$ if and only if $x$ and $y$ belong to the same $I\in\mathcal{D}^j$, but they do not belong to the same half of $I$. In other words, $(x,y)\in I_l\times I_r$ or $(x,y)\in I_r\times I_l$.

\noindent\textit{Proof of (b).} Fix $x\in \mathbb{R}^+$. Take $0<a<b<\infty$. Then, from \textit{(a)},
\begin{align*}
\int_{\{y\in B_\delta(x,b)\setminus B_\delta(x,a)\}}\delta^\alpha(x,y)dy
&=\int_{\{y: a\leq \delta(x,y)<b\}}\delta^\alpha(x,y)dy\\
&= \sum_{\{j\in \mathbb{Z}: a\leq 2^j<b\}}\int_{\{y:\delta(x,y)=2^j\}}2^{\alpha j}dy\\
&=\frac{1}{2}\sum_{\{j\in \mathbb{Z}: a\leq 2^j<b\}}2^{(1+\alpha)j}\\
&=\frac{1}{2}S(\alpha;a,b).
\end{align*}
When $\alpha\geq -1$, then $S(\alpha;a,b)\to +\infty$ for $b\to\infty$, for every $a$. Thus proves \textit{(iv)}. When $\alpha\leq -1$ then $S(\alpha;a,b)\to+\infty$ for $a\to 0$, for every $b$. For $\alpha>-1$, we have with $2^{j_0}\leq r<2^{j_0+1}$ that
\begin{equation*}
\int_{B_\delta(x,r)}\delta^\alpha(x,y)dy =\frac{1}{2}\lim_{a\to 0}S(\alpha;a,b)=\frac{1}{2}\sum_{j\leq j_0(r)}2^{(1+\alpha)j}
=\frac{1}{2}\frac{1}{1-2^{-(1+\alpha)}}2^{(1+\alpha)j_0}=c(\alpha)2^{(1+\alpha)j_0}.
\end{equation*}
Hence
\begin{equation*}
\frac{c(\alpha)}{2^{1+\alpha}}r^{1+\alpha}\leq \int_{y\in B_{\delta}(x,r)} \delta^{\alpha}(x,y) dy \leq c(\alpha)r^{1+\alpha}.
\end{equation*}
For $\alpha<-1$ we have, with $2^{j_0}\leq r<2^{j_0+1}$, that
\begin{equation*}
\int_{\delta(x,y)\geq r}\delta^{\alpha}(x,y)dy=\frac{1}{2}\lim_{b\to\infty}S(\alpha;r,b)=\frac{1}{2}\sum_{j\geq j_0(r)}(2^{1+\alpha})j=
\frac{1}{2}\frac{1}{1-2^{1+\alpha}}2^{(1+\alpha)j_0}=\tilde{c}(\alpha)2^{(1+\alpha)j_0},
\end{equation*}
so that
\begin{equation*}
\frac{\tilde{c}(\alpha)}{2^{1+\alpha}} r^{1+\alpha}\geq\int_{\{y: \delta(x,y)\geq r\}} \delta^{\alpha}(x,y) dy\geq\tilde{c}(\alpha)r^{1+\alpha}.
\end{equation*}
\end{proof}

The distance $\delta$ is not translation invariant. In fact, while for small positive $\varepsilon$, $\delta(\tfrac{1}{2}-\varepsilon,\tfrac{1}{2}+\varepsilon)=1$, $\delta(\tfrac{1}{2}+\tfrac{1}{2}-\varepsilon,\tfrac{1}{2}+\tfrac{1}{2}+\varepsilon)=2$. Neither is $\delta$ positively homogeneous. Nevertheless the next statement contains a useful property of dyadic homogeneity.

\begin{lemma}\label{lemma:deltahomogeneity}
Let $j\in \mathbb{Z}$ be given. Then, for $x$ and $y$ in $\mathbb{R}^+$, $\delta(2^jx,2^jy)=2^j\delta(x,y)$.
\end{lemma}
\begin{proof}
Notice first that since $x=y$ is equivalent to $2^jx=2^jy$, we may assume $x\neq y$. Since for $x$ and $y$ in $I\in \mathcal{D}$ we certainly have that $2^jx$ and $2^jy$ belong to $2^jI$, and the measure of $2^jI$ is $2^j$ times the measure of $I$, in order to prove the dyadic homogeneity of $\delta$, we only have to observe that the multiplication by $2^j$ as an operation on $\mathcal{D}$ preserves the order provided by inclusion. In particular $x$ and $y$ belong to $I$ but $x$ and $y$ do not belong to the same half $I_l$ or $I_r$ of $I$, if and only if $2^jx$ and $2^jy$ belong to $2^jI$ but $2^jx$ and $2^jy$ do not belong to the same half of $2^jI$.
\end{proof}

As in the classical case of the Central Limit Theorem, Fourier Analysis will play an important role in our
further development. The basic difference is that in our context the trigonometric expansions are substituted by the most
elementary wavelet analysis, the associated to the Haar system. Let us introduce the basic notation. Set $h^0_0(x)=\chi_{[0,1/2)}(x)-\chi_{[1/2,1)}(x)$
and, for $j\in \mathbb{Z}$ and $k=0,1,2,3,\ldots$; $h^j_k(x)=2^{j/2}h^0_0(2^jx-k)$. Notice that $h^j_k$ has $L^2$-norm equal to one for every $j$ and $k$. Moreover, $h^j_k$ is supported in $I=I^j_k\in \mathcal{D}^j$. Write $\mathscr{H}$ to denote the sequence of all those Haar wavelets. For $h\in\mathscr{H}$ we shall use the notation $I(h)$ to denote the interval $I$ in $\mathcal{D}$ for which $\supp h = I$. Also $j(h)$ is the only resolution level $j\in \mathbb{Z}$ such that $I(h)\in \mathcal{D}^j$.

The basic analytic fact of the system $\mathscr{H}$ is given by its basic character. In fact, $\mathscr{H}$ is an orthonormal basis for $L^2(\mathbb{R}^+)$. In particular,
for every $f\in L^2(\mathbb{R}^+)$ we have that in the $L^2$-sense $f=\sum_{h\in\mathscr{H}}\proin{f}{h}h$, where, as usual, for real valued $f$, $\proin{f}{h}=\int_{\mathbb{R}^+}f(x)h(x) dx$.

One of the most significant analytic properties of wavelets is its ability to characterize function spaces. For our purposes it will be useful  to have in mind the characterization of all $L^p(\mathbb{R}^+)$ spaces for $1<p<\infty$.

\begin{theorem}[Wojtaszczyk \cite{Wojtasbook}]\label{thm:characterizationLp}
For $1<p<\infty$ and some constants $C_1$ and $C_2$ we have
\begin{equation}
C_1\norm{f}_p\leq \norm{\left(\sum_{h\in\mathscr{H}}\abs{\proin{f}{h}}^2\abs{I(h)}^{-1}\chi_{I(h)}\right)^{1/2}}_p\leq C_2\norm{f}_p
\end{equation}
\end{theorem}

\section{Markov dyadic kernels defined in $\mathbb{R}^+$}\label{sec:Markovdyadickernels}
A real function $K$ defined in $\mathbb{R}^+\times \mathbb{R}^+$ is said to be a symmetric Markov kernel if $K$ is nonnegative, $K(x,y)=K(y,x)$ for every $x\in \mathbb{R}^+$ and $y\in \mathbb{R}^+$ and  $\int_{\mathbb{R}^+} K(x,y) dy=1$ for every $x\in \mathbb{R}^+$. We are interested in kernels $K$ as above such that $K(x,y)$ depends only on the dyadic distance $\delta(x,y)$ between the points $x$ and $y$ in $\mathbb{R}^+$. The next lemma contains three ways of writing such kernels $K$. The first is just a restatement of the dependence of $\delta$ and the other two shall be used frequently in our further analysis. The Lemma also includes relation between the coefficients and their basic properties.
\begin{lemma}\label{lemma:kerneldelta1}
Let $K$ be a real function defined on  $\mathbb{R}^+\times \mathbb{R}^+$. Assume that $K$ is nonnegative and depends only on $\delta$, i.e., $\delta(x,y)=\delta(x',y')$ implies $K(x,y)=K(x',y')$, with $\int_{\mathbb{R}^+} K(x_0,y)dy=1$ for some $x_0\in \mathbb{R}^+$. Then, with the  notation introduced in Lemma~1~(a) for the level sets of $\delta$, we have
\begin{enumerate}[(1)]
\item $K=\sum_{j\in \mathbb{Z}}k_j\chi_{L(2^j)}$, $k_j\geq 0$, $\sum_{j\in \mathbb{Z}}k_j2^{j-1}=1$ and $K$ is a symmetric Markov kernel.
\item The sequence $\overline{\alpha}=(\alpha_l=2^{-l}(k_{-l}-k_{-l+1}):l\in \mathbb{Z})$ belongs to $l^1(\mathbb{Z})$, $\sum_{l\in \mathbb{Z}}\alpha_l=1$ and the function $\varphi(s)=\sum_{l\in \mathbb{Z}}\alpha_l\varphi_l(s)$ with $\varphi_l(s)=2^{l}\chi_{(0,2^{-l}]}(s)$, provides a representation of $K$ in the sense that $\varphi(\delta(x,y))=K(x,y)$. Moreover, $\int_{\mathbb{R}^+}\abs{\varphi(s)}ds<\infty$ and $\int_{\mathbb{R}^+}\varphi(s)ds=1$.
\item The function $\varphi(s)$ can also be written as $\varphi(s)=\sum_{j\in \mathbb{Z}}\Lambda_j(\varphi_{j+1}(s)-\varphi_j(s))$.
\item\label{item:formulaerelated} The coefficients $\overline{k}=(k_j:j\in \mathbb{Z})$ in (1), $\overline{\alpha}=(\alpha_j:j\in \mathbb{Z})$ in (2) and $\overline{\Lambda}=(\Lambda_j:j\in \mathbb{Z})$ in (3) are related by the formulae
\begin{enumerate}[(\ref{item:formulaerelated}.a)]
\item $\alpha_j = \frac{k_{-j}-k_{-j+1}}{2^j}$

\item  $k_j = \sum_{i=j}^{\infty}2^{-i}\alpha_{-i}$

\item  $\Lambda_j = \sum_{l>j}\alpha_l $

\item $\alpha_j = \Lambda_{j-1}-\Lambda_j $

\item $\Lambda_j = \tfrac{1}{2}\left(-k_{-j}2^{-j}+\sum_{l<-j}k_l2^l\right) $

\item $k_j = -2^{-j}\Lambda_{-j}+\sum_{i\geq j+1}2^{-i}\Lambda_{-i}$.
\end{enumerate}

\item\label{item:propertiessequences} Some relevant properties of the sequences $\overline{k}$, $\overline{\alpha}$ and $\overline{\Lambda}$ are the following.
\begin{enumerate}[(\ref{item:propertiessequences}.a)]
\item $\overline{\alpha}\in l^1(\mathbb{Z})$;
\item $\sum_{l\leq j}\alpha_l2^l\geq 0$ for every $j\in \mathbb{Z}$;
\item $\abs{\alpha_l}\leq 2$ for every $l\in \mathbb{Z}$;
\item $\lim_{j\to-\infty}\Lambda_j=1$;
\item $\lim_{j\to+\infty}\Lambda_j=0$;
\item $\sum_{l\leq j-1}\Lambda_l2^l\geq\Lambda_j2^j$ for every $j\in \mathbb{Z}$;
\item $\sup_j\Lambda_j=1$;
\item $\inf_j\Lambda_j\geq -1$;
\item  if $\overline{k}$ is decreasing then also $\overline{\Lambda}$ is decreasing.
\end{enumerate}
\end{enumerate}
\end{lemma}
\begin{proof}[Proof of (1)]
Since $K$ depends only on $\delta$, then the level sets for $\delta$ are level sets for $K$. Hence $K$ is constant, say $k_j\geq 0$, in $L(2^j)$ for each $j\in \mathbb{Z}$. Notice that the section of $L(2^j)$ at any $x\in \mathbb{R}^+$ has measure $2^{j-1}$, no matter what is $x$. In fact, $\left. L(2^j)\right|_{x}=\{y\in \mathbb{R}^+:(x,y)\in L(2^j)\}=\{y\in \mathbb{R}^+:\delta(x,y)=2^j\}=I$, where $I\in\mathcal{D}$ is the brother of the dyadic interval $J$ of level $j-1$ such that $x\in J$. Hence $\abs{\left. L(2^j)\right|_{x}}=2^{j-1}$. With the above considerations, since $\int_{\mathbb{R}^+}K(x_0,y)dy=1$, we see that
\begin{align*}
1&=\int_{\mathbb{R}^+}K(x_0,y)dy=\sum_{j\in \mathbb{Z}}k_j\int_{\mathbb{R}^+}\chi_{L(2^j)}(x_0,y)dy\\
&=\sum_{j\in \mathbb{Z}}k_j\abs{\left. L(2^j)\right|_{x_0}}=\sum_{j\in \mathbb{Z}}k_j 2^{j-1}\\
&=\sum_{j\in \mathbb{Z}}k_j\abs{\left. L(2^j)\right|_{x}}=\int_{\mathbb{R}^+}K(x,y)dy.
\end{align*}
Then $K$ is a Markov kernel and that the series $\sum_{j\in \mathbb{Z}}k_j2^{j-1}$ converges to $1$. The symmetry of $K$ is clear.

\textit{Proof of (2).} Since $\abs{\alpha_l}\leq 2^{-l}k_{-l}+2^{-l}k_{-l+1}$, the fact that $\overline{\alpha}$ belongs to $l^1(\mathbb{Z})$ follow from the fact that $\sum_{j\in \mathbb{Z}}k_j2^j=2$ proved (1). On the other hand,
\begin{equation*}
\sum_{l\in \mathbb{Z}}\alpha_l=\sum_{l\in \mathbb{Z}}k_{-l}2^{-l}-\sum_{l\in \mathbb{Z}}k_{-l+1}2^{-l}=2-1=1.
\end{equation*}
Let us now check that $\varphi(\delta(x,y))=K(x,y)$. Since $\delta(x,y)$ is a integer power of two and $k_j\to 0$ as $j\to\infty$, we have
\begin{align*}
\varphi(\delta(x,y)) &=\sum_{l\in \mathbb{Z}}\alpha_l\varphi_l(\delta(x,y))\\
&= \sum_{l\in \mathbb{Z}}\alpha_l 2^l\chi_{(0,2^{-l}]}(\delta(x,y))\\
&=\sum_{l\leq\log_2\tfrac{1}{\delta(x,y)}}2^{-l}(k_{-l}-k_{-l+1})2^l\\
&= \sum_{j\geq\log_2\delta(x,y)}(k_j-k_{j+1})\\
&= k_{\log_2\delta(x,y)}=K(x,y).
\end{align*}
Now, the absolute integrability of $\varphi$ and the value of its integral follow from the formulae $\varphi(s)=\sum_{l\in \mathbb{Z}}\alpha_l\varphi_l(s)$ since $\overline{\alpha}\in l^1(\mathbb{Z})$, $\sum_{l\in \mathbb{Z}}\alpha_l=1$ and $\int_{\mathbb{R}^+}\varphi_l(s)ds=1$.

\smallskip
\textit{Proof of (3).} Fix a positive $s$ and proceed to sum by parts the series defining $\varphi(s)=\sum_{l\in \mathbb{Z}}\alpha_l\varphi_l(s)$. Set $\Lambda_j=\sum_{l>j}\alpha_l$. Since $\alpha_l=\Lambda_{l-1}-\Lambda_l$, we have that
\begin{equation*}
\varphi(s) = \sum_{l\in \mathbb{Z}}(\Lambda_{l-1}-\Lambda_l)\varphi_l(s)
=\sum_{l\in \mathbb{Z}}\Lambda_{l-1}\varphi_l(s)-\sum_{l\in \mathbb{Z}}\Lambda_{l}\varphi_l(s)
= \sum_{l\in \mathbb{Z}}\Lambda_{l}(\varphi_{l+1}(s)-\varphi_l(s)),
\end{equation*}
as desired. Notice, by the way, that $\varphi_{l+1}(s)-\varphi_l(s)$ can be written in terms of Haar functions as $\varphi_{l+1}(s)-\varphi_l(s)=2^{\tfrac{l}{2}}h^l_0(s)$.

\smallskip
\textit{Proof of (4).} Follows from the definitions of $\overline{\alpha}$ and $\overline{\Lambda}$.

\smallskip
\textit{Proof of (5).} Notice first that (5.a) was proved in (2). The nonnegativity of $K$ and (4.b) show (5.b). Property (5.d) and (5.e) of the sequence $\overline{\Lambda}$ follow from (4.c) and the fact that $\sum_{l\in \mathbb{Z}}\alpha_l=1$ proved in (2). Inequality (5.f) follows from the positivity of $K$ and (4.f).

We will prove (5.g). From  (5.d) and (5.e) we have that $\overline{\Lambda}\in l^{\infty}(\mathbb{Z})$. In fact, there exist $j_1<j_2$ in $\mathbb{Z}$ such that $\Lambda_j<2$ for $j<j_1$ and $\Lambda_j>-1$ for $j>j_2$. Since the set $\{\Lambda_{j_1},\Lambda_{j_1+1},\ldots,\Lambda_{j_2}\}$ is finite, we get the boundedness of $\overline{\Lambda}$. On the other hand, since from (5.d) $\lim_{j\to-\infty}\Lambda_j=1$ we have that $\sup_j\Lambda_j\geq 1$. Assume that $\sup_j\Lambda_j> 1$. Then there exists $j_0\in \mathbb{Z}$ such that $\Lambda_{j_0}>1$. Hence, again from (5.d) and (5.e) we must have that for $j<j_3$, $\Lambda_j<\Lambda_{j_0}$ and for $j>j_4$, $\Lambda_j<1<\Lambda_{j_0}$ for some integers $j_3<j_4$. So that there exists $j_5\in \mathbb{Z}$ such that $\Lambda_{j_5}\geq\Lambda_j$ for every $j\in \mathbb{Z}$ and $\Lambda_{j_5}>1$. Now
\begin{equation*}
2^{j_5}\Lambda_{j_5}=\sum_{l\leq j_5 -1}\Lambda_{j_5}2^l>\sum_{l\leq j_5-1}\Lambda_l2^l
\end{equation*}
which contradicts (5.f) with $j=j_5$.

For prove (5.h) assume that $\inf_j\Lambda_j<-1$. Choose $j_0\in \mathbb{Z}$ such that $\Lambda_{j_0}<-1$. Then from (5.f)
\begin{equation*}
\Lambda_{j_0+1}\leq 2^{-(j_0+1)}\sum_{l\leq j_0}\Lambda_l2^l
=\sum_{l\leq j_0}\Lambda_l2^{l-(j_0+1)}
=\frac{1}{2}\left(\Lambda_{j_0}+\sum_{l< j_0}\Lambda_l2^{l-j_0)}\right)
\leq\frac{1}{2}(\Lambda_{j_0}+1).
\end{equation*}
In the last inequality we used (5.g). Let us prove, inductively, that $\Lambda_{j_0+m}\leq\tfrac{1}{2}(\Lambda_{j_0}+1)$ for every $m\in \mathbb{N}$. Assume that the above inequality holds for $1\leq m\leq m_0$ and let us prove it for  $m_0+1$.
\begin{align*}
\Lambda_{j_0+(m_0+1)}&\leq \sum_{l<j_0+m_0+1}2^{l-(j_0+m_0+1)}\Lambda_l\\
&=2^{-m_0-1}\left(\sum_{l=j_0}^{j_0+m_0}2^{l-j_0}\Lambda_l+\sum_{l<j_0}2^{l-j_0}\Lambda_l\right)\\
&=2^{-m_0-1}\left(\sum_{l=1}^{m_0}2^{l}\Lambda_{j_0+l}+\Lambda_{j_0}+\sum_{l<j_0}2^{l-j_0}\Lambda_l\right)\\
&\leq 2^{-m_0-1}\left(\sum_{l=1}^{m_0}2^{l-1}(\Lambda_{j_0}+1)+\Lambda_{j_0}+\sum_{l<j_0}2^{l-j_0}\right)\\
&=2^{-m_0-1}((2^{m_0}-1)(\Lambda_{j_0}+1)+\Lambda_{j_0}+1)\\
&=\frac{1}{2}(\Lambda_{j_0}+1).
\end{align*}

Property (5.c) for the sequence $\overline{\alpha}$ follows from (4.d), (5.g) and (5.h). Item (5.i) follows from (4.a) and (4.d).
\end{proof}

In the sequel we shall write $\mathscr{K}$ to denote the set of all nonnegative kernels defined on $\mathbb{R}^+\times \mathbb{R}^+$ that depends only on $\delta$ and for some $x_0\in \mathbb{R}^+$, $\int_{\mathbb{R}^+}K(x_0,y)dy=1$.

Let us finish this section by proving a lemma that shall be used later.

\begin{lemma}\label{lemma:basiccharacterizationK}
Let $\overline{\Lambda}=(\Lambda_j:j\in \mathbb{Z})$ be a decreasing sequence of real numbers satisfying
(5.d) and (5.e). Then there exists a unique $K\in\mathscr{K}$ such that the sequence that (3) of  Lemma~\ref{lemma:kerneldelta1} associates to $K$ is the given $\overline{\Lambda}$.
\end{lemma}

\begin{proof}
Define $K(x,y)=\sum_{j\in \mathbb{Z}}(\Lambda_{j-1}-\Lambda_j)\varphi_j(\delta(x,y))$. Since $\overline{\Lambda}$ is decreasing the coefficients in the above series are all nonnegative. On the other hand, from (5.d) and (5.e) we have that $\sum_{j\in \mathbb{Z}}(\Lambda_{j-1}-\Lambda_j)=1$. Hence, for every $x\in \mathbb{R}^+$ we have
\begin{equation*}
\int_{y\in \mathbb{R}^+}K(x,y)dy = \sum_{j\in \mathbb{Z}}(\Lambda_{j-1}-\Lambda_j)\int_{y\in \mathbb{R}^+}\varphi_j(\delta(x,y))dy
= \sum_{j\in \mathbb{Z}}(\Lambda_{j-1}-\Lambda_j)=1
\end{equation*}
So that $K\in\mathscr{K}$.
\end{proof}

\section{The spectral analysis of the operators induced by kernels in $\mathscr{K}$}\label{sec:spectralanalysis}
For $K\in\mathscr{K}$ and $f$ continuous with bounded support in $\mathbb{R}^+$ the integral $\int_{\mathbb{R}^+}K(x,y)f(y)dy$ is well defined and finite for each $x\in \mathbb{R}^+$. Actually each $K\in\mathscr{K}$ determines an operator which is well defined and bounded on each $L^p(\mathbb{R}^+)$ for $1\leq p\leq\infty$.

\begin{lemma}
Let $K\in\mathscr{K}$ be given. Then for $f\in L^p(\mathbb{R}^+)$ the integral $\int_{\mathbb{R}^+}K(x,y)f(y)dy$ is absolutely convergent for almost every $x\in \mathbb{R}^+$. Moreover,
\begin{equation*}
Tf(x)=\int_{\mathbb{R}^+}K(x,y)f(y) dy
\end{equation*}
defines a bounded (non-expansive) operator on each $L^p(\mathbb{R}^+)$, $1\leq p\leq\infty$. Precisely, $\norm{Tf}_p\leq\norm{f}_p$ for $f\in L^p(\mathbb{R}^+)$.
\end{lemma}
\begin{proof}
Notice first that the function $K(x,y)f(y)=\varphi(\delta(x,y))f(y)$ is measurable as a function defined on $\mathbb{R}^+ \times\mathbb{R}^+$, for every measurable $f$ defined on $\mathbb{R}^+$. The case $p=\infty$ follows directly from the facts that $K$ is a Markov kernel and that $K(x,y)\abs{f(y)}\leq K(x,y)\norm{f}_\infty$. For $p=1$ using Tonelli's theorem we get
\begin{equation*}
\int_{x\in \mathbb{R}^+}\left(\int_{y\in \mathbb{R}^+}K(x,y)\abs{f(y)}dy\right)dx=
\int_{y\in \mathbb{R}^+}\abs{f(y)}\left(\int_{x\in \mathbb{R}^+}K(x,y)dx\right)dy=\norm{f}_1.
\end{equation*}
Hence $\int_{\mathbb{R}^+}K(x,y)f(y) dy$ is absolutely convergent for almost every $x$ and $\norm{Tf}_1\leq\norm{f}_1$. Assume that $1<p<\infty$ and take $f\in L^p(\mathbb{R}^+)$. Then
\begin{align*}
\abs{Tf(x)}^p &\leq \left(\int_{\mathbb{R}^+}K(x,y)\abs{f(y)} dy\right)^p=\left(\int_{\mathbb{R}^+}K(x,y)^{\tfrac{1}{p'}}K(x,y)^{\tfrac{1}{p}}\abs{f(y)} dy\right)^p\\
&\leq \left(\int_{\mathbb{R}^+}K(x,y) dy\right)^{\tfrac{p}{p'}}\left(\int_{\mathbb{R}^+}K(x,y)\abs{f(y)}^p dy\right)\\
&=\int_{\mathbb{R}^+}K(x,y)\abs{f(y)}^p dy.
\end{align*}
Hence $\norm{Tf}^p_p=\int_{\mathbb{R}^+}\abs{Tf(x)}^p dx\leq \int_{y\in\mathbb{R}^+}\left(\int_{x\in\mathbb{R}^+}K(x,y) dx\right)\abs{f(y)}^p dy=\norm{f}^p_p$.
\end{proof}

The spectral analysis of the operators $T$ defined by kernels in $\mathscr{K}$ is given in the next result.

\begin{theorem}\label{thm:autovalores}
Let $K\in\mathscr{K}$ and let $T$ be the operator in $L^2(\mathbb{R}^+)$ defined by $Tf(x)=\int_{\mathbb{R}^+}K(x,y)f(y) dy$. Then the Haar functions are eigenfunctions for $T$ and the eigenvalues are given by the sequence $\overline{\Lambda}$ introduced in Lemma~\ref{lemma:kerneldelta1}. Precisely, for each $h\in\mathscr{H}$
\begin{equation*}
Th=\Lambda_{j(h)}h:=\lambda(h) h,
\end{equation*}
where $j(h)$ is the level of the support of $h$, i.e. $supp\, h\in\mathcal{D}^{j(h)}$.
\end{theorem}
\begin{proof}
Since the sequence $(\alpha_l:l\in \mathbb{Z})$ belongs to $\ell^1(\mathbb{Z})$ and we can interchange orders of integration and summation
in order to compute $Th$. In fact,
\begin{equation*}
T h(x) = \int_{y\in\mathbb{R}^+} \varphi(\delta(x,y))h(y) dy
=\int_{y\in\mathbb{R}^+}\left(\sum_{l\in \mathbb{Z}}\alpha_l\varphi_l(\delta(x,y))\right) h(y) dy
= \sum_{l\in \mathbb{Z}}\alpha_l\left(2^{l}\int_{\{y: \delta(x,y)\leq 2^{-l}\}}h(y)dy\right).
\end{equation*}
Let us prove that
\begin{equation*}
\psi(x,l)=2^{l}\int_{\{y: \delta(x,y)\leq 2^{-l}\}}h(y) dy=\chi_{\{l>j(h)\}}(l) h(x).
\end{equation*}
If $x\notin I(h)$, since $\{y:\delta(x,y)\leq 2^l\}$ is the only dyadic interval $I_l^x$ containing $x$ of length $2^l$, only two situations are possible, $I_l^x\cap I(h)=\emptyset$ or $I_l^x\supset I(h)$, in both cases the integral vanish and $\psi(x,l)=0=\chi_{\{l<-j(h)\}}(l) h(x)$. Take now $x\in I(h)$. Assume first that $x\in I_l(h)$ (the left half of $I(h)$). So that $\psi(x,l)=2^{-l}\int_{I_l^x}h(y) dy=0$ if $l\leq j(h)$, since $I_l^x\supset I(h)$. When $l>j(h)$ we have that $h\equiv \abs{I(h)}^{-1/2}$ on $I_l^x$, hence $\psi(l,x)=2^{-l}\abs{I(h)}^{-1/2}\abs{I_l^x}=\abs{I(h)}^{-1/2}=h(x)$. In a similar way, for $x\in I_r(h)$, we get $\psi(l,x)=-\abs{I(h)}^{-1/2}=h(x)$.
\end{proof}

Notice that the eigenvalues $\lambda(h)$ tends to zero when the resolution $j(h)$ tends to infinity. Moreover this convergence is monotonic when all the $\alpha_l$ are nonnegative. Notice also that the eigenvalues depend only on the resolution level of $h$, but not on the position $k$ of its support. Sometimes we shall write $\lambda_j$, $j\in \mathbb{Z}$, instead of $\lambda(h)$ when  $j$ is the scale of the support of $h$. With the above result, and using the fact that the Haar system $\mathscr{H}$ is an orthonormal basis for $L^2(\mathbb{R}^+)$, we see that, the action of $T$ on $L^2(\mathbb{R}^+)$ can be regarded as a multiplier operator on the scales.

\begin{lemma}
Let $K$ and $T$ as in Theorem~\ref{thm:autovalores}. The diagram

\begin{center}
\begin{tikzpicture}
  \matrix (m) [matrix of math nodes,row sep=3em,column sep=4em,minimum width=2em]
  {
   L^2(\mathbb{R}^+) & \ell^2(\mathbb{Z}) \\
     L^2(\mathbb{R}^+) & \ell^2(\mathbb{Z}) \\};
  \path[-stealth]
    (m-1-1) edge node [left] {$T$} (m-2-1)
            edge node [below] {$H$} (m-1-2)
    (m-2-1.east|-m-2-2) edge node [below] {$H$}
            node [above] {} (m-2-2)
    (m-1-2) edge node [right] {$M$} (m-2-2);
\end{tikzpicture}
\end{center}
commutes, where $H(f)=(\proin{f}{h}: h\in\mathscr{H})$ and $M(a_h:h\in\mathscr{H})=(\lambda(h)a_h:h\in\mathscr{H})$. In particular, $\norm{Tf}^2_2=\sum_{h\in\mathscr{H}}
\lambda^2(h)\abs{\proin{f}{h}}^2$.
\end{lemma}
The characterization of the space $L^p(\mathbb{R}^+)$ ($1<p<\infty$), Theorem~\ref{thm:characterizationLp} above,  provides a similar result for the whole scale of Lebesgue spaces, $1<p<\infty$ with the only caveat that when $p\neq 2$ the norms are only equivalent. The next statement contains this observation.
\begin{theorem}\label{thm:op.Lp.haar}
With $K$ and $T$ as before and $1<p<\infty$ we have that
\begin{equation*}
\norm{Tf}_p\simeq\norm{\biggl(\sum_{h\in\mathscr{H}}(\lambda(h))^2\abs{\proin{f}{h}}^2\abs{I(h)}^{-1}\chi_{I(h)}\biggr)^{\tfrac{1}{2}}}_p
\end{equation*}
with constants which do not depend on $f$.
\end{theorem}

\begin{corollary}\label{coro:representationK}
For every $K\in\mathscr{K}$ and $(\lambda(h):h\in\mathscr{H})$ as in Theorem~\ref{thm:autovalores} we have the representation
\begin{equation*}
K(x,y)=\sum_{h\in\mathscr{H}}\lambda(h)h(x)h(y).
\end{equation*}
\end{corollary}
\begin{proof}
For $f=\sum_{h\in\mathscr{H}}\proin{f}{h}h$ with $\proin{f}{h}\neq 0$ only for finitely many Haar functions $h\in\mathscr{H}$, we have that
\begin{align*}
\int_{\mathbb{R}^+}K(x,y)f(y)dy=Tf(x)&=\sum_{h\in\mathscr{H}}\proin{f}{h}Th(x)\\
&=\sum_{h\in\mathscr{H}}\left(\int_{y\in\mathbb{R}^+}f(y)h(y)dy\right)\lambda(h)h(x)\\
&=\int_{y\in\mathbb{R}^+}\left(\sum_{h\in\mathscr{H}}\lambda(h)h(y)h(x)\right)f(y)dy.
\end{align*}
Since the space of such functions $f$ is dense in $L^2(\mathbb{R}^+)$ we have that $K(x,y)=\sum_h\lambda(h)h(x)h(y)$.
\end{proof}

\section{Stability of Markov kernels}\label{sec:stability}

In the case of the classical CLT the key properties of the distribution of the independent random variables $X_j$ are contained in the Gaussian central limit itself. Precisely, $(2\pi t)^{-1/2}e^{-\abs{x}^2/4t}$ is the distribution limit of $n^{-1/2}\sum_{j=1}^n X_j$ when $X_j$ are independent and are equi-distributed with variance $t$ and mean zero. Our ``gaussian'' is the Markov kernel $K_t(x,y)$ defined in $\mathbb{R}^+\times \mathbb{R}^+$ by applying Lemma~\ref{lemma:basiccharacterizationK} to the sequence $\Lambda_j=e^{-t2^{j}}$, $j\in \mathbb{Z}$ for fixed $t$. We may also use the Haar representation of $K_t(x,y)$ given by Corollary~\ref{coro:representationK} in \S~\ref{sec:spectralanalysis}. In this way we can write this family of kernels as $K_t(x,y)=\sum_{h\in\mathscr{H}}e^{-t2^{j(h)}}h(x)h(y)$. As we shall see, after obtaining estimates for the behavior of $K$ for large $\delta(x,y)$, this kernel has heavy tails. In particular, the analogous of the variance given by $\int_{y\in \mathbb{R}^+}K_t(x,y)\delta^2(x,y)dy$ is not finite. This kernel looks more as a dyadic version of Cauchy type distributions than of Gauss type distributions. Which is an agreement with the fact that $K_t$ solves a fractional differential equation and the natural processes are of Lévy type instead of Wiener Brownian. As a consequence, the classic moment conditions have to be substituted by stability type behavior at infinity.

\begin{lemma}\label{lemma:gaussianPsistability23}
Set for $r>0$
\begin{equation*}
\psi(r)=\frac{1}{r}\left(\sum_{j\geq 1}2^{-j}e^{-(2^jr)^{-1}}-e^{-r^{-1}}\right).
\end{equation*}
Then $\psi$ is well defined on $\mathbb{R}^+$ with values in $\mathbb{R}^+$. And
\begin{equation*}
r^{2}\psi(r)\to \frac{2}{3} \textrm{\quad as \quad} r\to\infty.
\end{equation*}
\end{lemma}
\begin{proof}
Since $e^{-(2^jr)^{-1}}$ is bounded above we see that $\psi(r)$ is finite for every $r>0$. On the other hand since
$\psi(r)=\tfrac{1}{r}\sum_{j\geq 1}2^{-j}[e^{-(2^jr)^{-1}}-e^{-r^{-1}}]$ and terms in brackets are positive we see that $\psi(r)>0$ for every $r>0$. Let us check the behavior of $\psi$ at infinity
\begin{equation*}
r^{2}\psi(r)=\sum_{j\geq 1}\frac{2^{-j}[e^{-(2^jr)^{-1}}-e^{-r^{-1}}]}{r^{-1}}\to \sum_{j\geq 1}2^{-j}(1-2^{-j})=\frac{2}{3}.
\end{equation*}
\end{proof}

\begin{lemma}\label{lemma:stability23}
Let $t>0$ be given. Set $\Lambda^{(t)}_j=e^{-t2^{j}}$, $j\in \mathbb{Z}$. Let $K_t(x,y)$ be the kernel that Lemma~\ref{lemma:basiccharacterizationK} associated to $\overline{\Lambda^{(t)}}$. Then $K_t\in\mathscr{K}$ and since $K_t(x,y)=\tfrac{1}{t}\psi(\tfrac{\delta(x,y)}{t})$, with $\psi$ as in Lemma~\ref{lemma:gaussianPsistability23}, we have
\begin{equation}\label{eq:propertystabilityone}
\delta(x,y)^{2}K_t(x,y)\to \frac{2}{3}\,t
\end{equation}
for $\delta(x,y)\to+\infty$.
\end{lemma}
\begin{proof}
Since $\Lambda^{(t)}_{j+1}<\Lambda^{(t)}_{j}$, for every $j\in \mathbb{Z}$, $\lim_{j\to-\infty}\Lambda^{(t)}_{j}=1$ and $\lim_{j\to +\infty}\Lambda^{(t)}_{j}=0$ we can use Lemma~\ref{lemma:basiccharacterizationK} in order to obtain the kernel $K_t(x,y)$. Now from Corollary~\ref{coro:representationK} we have that $K_t(x,y)=\sum_{h\in\mathscr{H}}e^{-t2^{j}}h(x)h(y)$. Let us check following the lines of \cite{AcAimFCAA}, that $K_t(x,y)=\tfrac{1}{t}\psi(\tfrac{\delta(x,y)}{t})$, with $\psi$ as in Lemma~\ref{lemma:gaussianPsistability23}. In fact, since $K_t(x,y)=\sum_{h\in\mathscr{H}}e^{-t\abs{I(h)}^{-1}}h(x)h(y)$, then a Haar function $h\in\mathscr{H}$ contributes to the sum when $x$ and $y$ both belong to $I(h)$. The smallest of such intervals, say $I_0=I(h^{(0)})$ is precisely the dyadic interval that determines $\delta(x,y)$. Precisely $\abs{I_0}=\delta(x,y)$. Let $h^{(1)}$ and $I_1=I(h^{(1)})$ be the wavelet and its dyadic support corresponding to one level less of resolution than that $I_0$ itself. In more familiar terms, $I_0$ is one of two son of $I_1$. In general, for each resolution level less than that of $I_0$ we find one and only one $I_i=I(h^{(i)})$ with $I_0\subset I_1\subset\ldots\subset I_i\subset\ldots$ and $\abs{I_i}=2^i\abs{I_0}$. We have to observe that except for $I_0$ where $x$ and $y$ must belong to different halves $I_{0,r}$ or $I_{0,l}$ of $I_0$, because of the minimality of $I_0$ for all the other $I_i$, $x$ and $y$ must belong to the same half $I_{i,l}$ or $I_{i,r}$ of $I_i$ because they are all dyadic intervals. These properties also show that $h^{(0)}(x)h^{(0)}(y)=-\abs{I_0}^{-1}=-\delta^{-1}(x,y)$ and, for $i\geq 1$, $h^{(i)}(x)h^{(i)}(y)=2^{-i}\abs{I_0}^{-1}=(2^i\delta(x,y))^{-1}$. Hence
\begin{align*}
K_t(x,y) &= -\frac{e^{-\tfrac{t}{\delta(x,y)}}}{\delta(x,y)}+\sum_{i\geq 1}e^{-\tfrac{t2^{-i}}{\delta(x,y)}}\frac{2^{-i}}{\delta(x,y)}\\
&= \frac{1}{\delta(x,y)}\left[\sum_{i\geq 1}2^{-i}e^{-\tfrac{t}{\delta(x,y)}2^{-i}}-e^{-\tfrac{t}{\delta(x,y)}}\right]\\
&= \frac{1}{t}\psi\left(\frac{\delta(x,y)}{t}\right).
\end{align*}
So that
\begin{equation*}
\delta(x,y)^{2}K_t(x,y)=\delta(x,y)^{2}\frac{1}{t}\psi\left(\frac{\delta(x,y)}{t}\right)
=t\left(\frac{\delta(x,y)}{t}\right)^{2}\psi\left(\frac{\delta(x,y)}{t}\right)
\end{equation*}
which from the result of Lemma~\ref{lemma:gaussianPsistability23} tends to $\tfrac{2}{3}$ when $\delta(x,y)\to +\infty$.
\end{proof}

Notice that from Lemma~\ref{lemma:deltaintegrability}-\textit{b.iv)} and the behavior at infinity of $K_t(x,y)$ provided in the previous result, we have
\begin{equation*}
\int_{R^+}K_t(x,y)\delta^2(x,y)dy=+\infty
\end{equation*}
for every $x\in \mathbb{R}^+$. Moreover, $\int_{R^+}K_t(x,y)\delta(x,y)dy=+\infty$. The adequate substitute for the property of finiteness of moments is provided by the stability involved in property \eqref{eq:propertystabilityone} in Lemma~\ref{lemma:stability23}. Since this property is going to be crucial in our main result we introduce formally the concept of stability. We say that a kernel $K$ in $\mathscr{K}$ is \textbf{\boldmath{$1$}-stable with parameter \boldmath{$\sigma>0$}} if
\begin{equation*}
\delta(x,y)^2 K(x,y)\to \sigma
\end{equation*}
for $\delta(x,y)\to\infty$. In the above limit, since the dimension of $\mathbb{R}^+$ with the metric $\delta$ equals one, we think $\delta^2$ as $\delta^{1+1}$, one for the dimension and the other for the order of stability.

Since for $K\in\mathscr{K}$ we have $K(x,y)=\varphi(\delta(x,y))$, the property of $1$-stability can be written as a condition for the behavior at infinity of profile $\varphi$. In particular, with the notation of Lemma~\ref{lemma:kerneldelta1}, the stability is equivalent to $4^jk_j\to\sigma$ as $j\to\infty$.

\section{Iteration and mollification in $\mathscr{K}$}\label{sec:iterationmollification}
As we have already observed in the introduction, the two basic operations on the  identically distributed independent random variables $X_i$ in order to obtain the means that converge in distribution to the Central Limit, translate into iterated convolution and mollification. In this section, we shall be concerned with two operations, iteration and mollification on $\mathscr{K}$ and on the subfamily $\mathscr{K}^1$ of $1$-stable kernels in $\mathscr{K}$.

In the sequel, given a kernel $K$ in $\mathscr{K}$, $\bar{\Lambda}$, $\bar{\alpha}$ and $\bar{k}$ are the sequences defined on Lemma~\ref{lemma:kerneldelta1} associated to $K$. When a  family of kernels in $\mathscr{K}$ is described by an index associated to $K$, say $K_i$, the corresponding sequences are denoted by $\bar{\Lambda}^i$, $\bar{\alpha}^i$ and $\bar{k}^i$.
\begin{lemma}
\begin{enumerate}[(a)]
\item For $K_1$ and $K_2\in\mathscr{K}$, the kernel
$$K_3(x,y)=(K_1\ast K_2)(x,y)=\int_{z\in \mathbb{R}^+}K_1(x,z)K_2(z,y)dz$$ is well defined; $K_3\in\mathscr{K}$ with
\begin{equation*}
\alpha^3_j=\alpha^1_j\lambda^2_j+\alpha^2_j\lambda^1_j+\alpha^1_j\alpha^2_j
\end{equation*}
for every $j\in \mathbb{Z}$;
\item $(\mathscr{K},\ast)$ and $(\mathscr{K}^1,\ast)$ are semigroups;
\item $\lambda^3_j=\lambda^1_j\lambda^2_j$ for every $j\in \mathbb{Z}$.
\end{enumerate}
\end{lemma}
\begin{proof}[Proof of (a)]
Let $K_i(x,y)=\varphi^i(\delta(x,y))$, $i=1,2$; with $\varphi^i(s)=\sum_{j\in \mathbb{Z}}\alpha^i_j\varphi_j(s)$, $\sum_{j\in \mathbb{Z}}\alpha^i_j=1$, $\sum_{j\in \mathbb{Z}}\abs{\alpha^i_j}<\infty$. Then, for $x\neq y$ both in $\mathbb{R}^+$. Set $I^*$ to denote the smallest dyadic interval containing $x$ and $y$. Then $\abs{I^*}=\delta(x,y)$ and $x$ and $y$ belong to different halves of $I^*$. From the above properties of the sequences $\bar{\alpha}^i$, $i=1,2$; we can interchange the orders of summation and integration in order to obtain
\begin{align*}
K_3(x,y) &= \int_{z\in \mathbb{R}^+}K_1(x,z)K_2(z,y)dz\\
&=\sum_{j\in \mathbb{Z}}\sum_{l\in \mathbb{Z}}2^i\alpha^1_j2^l\alpha^2_l\int_{z\in \mathbb{R}^+}\chi_{(0,2^{-j}]}(\delta(x,z))\chi_{(0,2^{-l}]}(\delta(z,y))dz\\
&=\sum_{j\in \mathbb{Z}}2^j\alpha^1_j\sum_{l\in \mathbb{Z}}2^l\alpha^2_l\abs{I^j_{k(x)}\cap I^l_{k(y)}}
\end{align*}
where $I^j_{k(x)}$ is the only dyadic interval in $\mathcal{D}^j$ such that $x\in I^j_{k(x)}$. Notice that the intersection of $I^j_{k(x)}$ and $I^l_{k(y)}$ is empty when $j$ and $l$ are both larger than the level $j^*$ of $I^*$. On the other hand, when $j$ or $l$ is smaller than or equal to $j^*$, the intersection is the smallest one. Say, if $j\leq j^*$ and $l>j$, $I^j_{k(x)}\cap I^l_{k(y)}=I^l_{k(y)}$.

With the above considerations we are now in position to compute $K_3(x,y)$ in terms of the sequences $\bar{\alpha}^i$ and $\bar{\lambda}^i$ as follows, with $c(j^*)=\{(j,l)\in \mathbb{Z}^2:j>j^* \textrm{ and } l>j^*\}$,
\begin{align*}
K_3(x,y) &= \sum\sum_{(j,l)\in \mathbb{Z}^2}2^{j+l}\alpha^1_j\alpha^2_l\abs{I^j_{k(x)}\cap I^l_{k(y)}}\\
&= \sum\sum_{\mathbb{Z}^2\setminus c(j^*)}2^{j+l}\alpha^1_j\alpha^2_l\abs{I^j_{k(x)}\cap I^l_{k(y)}}\\
&= \sum_{j\leq j^*}2^j\alpha^1_j\sum_{l>j}2^l\alpha^2_l\abs{I^l_{k(y)}}
+ \sum_{l\leq j^*}2^l\alpha^2_l\sum_{j>l}2^j\alpha^1_j\abs{I^j_{k(x)}}
+ \sum_{l\leq j^*}2^l\alpha^2_l2^l\alpha^1_l\abs{I^l_{k(y)}}\\
&= \sum_{j\leq j^*}2^j\alpha^1_j\lambda^2_j+\sum_{l\leq j^*}2^l\alpha^2_l\lambda^1_l+\sum_{l\leq j^*}2^l\alpha^1_l\alpha^2_l\\
&= \sum_{j\leq j^*}\left[\alpha^1_j\lambda^2_j+\alpha^2_j\lambda^1_j+\alpha^1_j\alpha^2_j\right]2^j\\
&= \sum_{j\in \mathbb{Z}}\left[\alpha^1_j\lambda^2_j+\alpha^2_j\lambda^1_j+\alpha^1_j\alpha^2_j\right]\varphi_j(\delta(x,y)).
\end{align*}
In other words, $K_3(x,y)=\varphi^3(\delta(x,y))$ with $\varphi^3(s)=\sum_{j\in \mathbb{Z}}\alpha^3_j\varphi_j(S)$
and $\alpha^3_j=\alpha^1_j\lambda^2_j+\alpha^2_j\lambda^1_j+\alpha^1_j\alpha^2_j$. Since, as it is easy to check by Tonelli's theorem $\int_{\mathbb{R}^+} K_3(x,y)dy=1$, we have that $K_3\in\mathscr{K}$.

\textit{Proof of (b).}  We only have to show that if $K_1$ and $K_2$ are $1$-stable kernels in $\mathscr{K}$, then $K_3=K_1\ast K_2$ is also $1$-stable. As we observed at the end of Section~\ref{sec:stability} for $K_i$ $(i=1,2)$ we have $4^jk^i_j\to\sigma_i$ when $j\to+\infty$. We have to prove that $4^jk^3_j\to \sigma_1+\sigma_2$ when $j\to+\infty$. By Lemma~\ref{lemma:kerneldelta1}, item (4.b), we can write
\begin{align*}
4^jk^3_j &= 4^j\sum_{i\geq j}2^{-i}\alpha^3_{-i}\\
&= 4^j\sum_{i\geq j}2^{-i}[\alpha^1_{-i}\lambda^2_{-i}+\alpha^2_{-i}\lambda^1_{-i}+\alpha^1_{-i}\alpha^2_{-i}]\\
&= 4^j\sum_{i\geq j}(2^{-i}\alpha^1_{-i})\lambda^2_{-i}+4^j\sum_{i\geq j}(2^{-i}\alpha^2_{-i})\lambda^1_{-i}+
4^j\sum_{i\geq j}2^{-i}\alpha^1_{-i}\alpha^2_{-i}\\
&= I(j)+II(j)+III(j).
\end{align*}
We claim that $I(j)\to\sigma_1$, $II(j)\to\sigma_2$ and $III(j)\to 0$ when $j\to+\infty$. Let us prove that $I(j)\to\sigma_1$, $j\to +\infty$. Since
\begin{equation*}
\abs{I(j)-\sigma_1}\leq \abs{4^j\sum_{i\geq j}2^{-i}\alpha^1_{-i}(\lambda^2_{-i}-1)}+\abs{4^jk^1_j-\sigma_1}
\end{equation*}
from the fact that $K_1\in\mathscr{K}^1$ with parameter $\sigma_1$ and because of (5.d) in Lemma~\ref{lemma:kerneldelta1} we have that $I(j)\to\sigma_1$ as $j\to\infty$. The fact $II(j)\to\sigma_2$ follows the same pattern. Let us finally estimate $III(j)$. Notice that from (4.a) en Lemma~\ref{lemma:kerneldelta1} we have
\begin{align*}
\abs{III(j)}&\leq 4^j\sum_{i\geq j}2^{-i}\abs{\alpha^1_{-i}}\abs{\alpha^2_{-i}}\\
&\leq 4^j\left(\sum_{i\geq j}2^{-i}\abs{\alpha^1_{-i}}\right)\left(\sum_{l\geq j}\abs{\alpha^2_{-l}}\right)\\
&= 4^j\left(\sup_{i\geq j}2^{-i}\abs{\frac{k^1_i-k^1_{i+1}}{2^{-i}}}\right)\left(\sum_{l\geq j}\abs{\alpha^2_{-l}}\right)\\
&\leq 2\,4^j\sup_{i\geq j}k^1_i\left(\sum_{l\geq j}\abs{\alpha^2_{-l}}\right)\\
&= 2\,4^j k^1_{i(j)}\left(\sum_{l\geq j}\abs{\alpha^2_{-l}}\right),
\end{align*}
where, since $k_i\to 0$ when $j\to \infty$, $i(j)\geq j$ is the necessarily attained supremum of the $k_i$'s for $i\geq j$. So that $4^jk^1_{i(j)}=4^{j-i(j)}4^{i(j)}k^1_{i(j)}$ is bounded above because $K_1\in\mathscr{K}^1$. On the other hand, since $\bar{\alpha}^2\in l^1(\mathbb{Z})$ the tail $\sum_{l\geq j}\abs{\alpha^2_{-l}}$ tends to zero as $j\to\infty$.

\textit{Proof of (c).} Since each $K_i$, $i=1,2$, can be regarded as the kernel of the operator $T_if(x)=\int_{y\in \mathbb{R}^+}K_i(x,y)f(y)dy$, $K_3$ is the kernel of the composition of $T_1$ and $T_2$, we have that
\begin{equation*}
T_3h=(T_2\circ T_1)h=T_2(T_1h)=T_2(\lambda^1(h)h)=\lambda^1(h)T_2h=\lambda^1(h)\lambda^2(h)h.
\end{equation*}
So $\lambda^1$ and $\lambda^2$ depend only on the scale $j$ of $h$, so does $\lambda^3=\lambda^1\lambda^2$.
\end{proof}

\begin{corollary}
Let $K\in\mathscr{K}^1$ with parameter $\sigma$, then for $n$ positive integer the kernel $K^n$ obtained as the composition of $K$ $n$-times, i.e.,
\begin{equation*}\label{coro:compositonKntimes}
K^{(n)}(x,y)=\idotsint_{(\mathbb{R}^+)^{n-1}}K(x,y_1)\cdots K(y_{n-1},y)dy_1\cdots dy_{n-1}
\end{equation*}
belongs to $\mathscr{K}^1$ with parameter $n\sigma$ and eigenvalues $\lambda^{(n)}_j=(\lambda_j)^n$, $j\in \mathbb{Z}$, with $\lambda_j$ the eigenvalues of $K$.
\end{corollary}

Trying to keep the analogy with the classical CLT, the mollification operator, that we have to define, is expected to preserve $\mathscr{K}^1$ producing a contraction of the parameter $\sigma$ in order to counteract the dilation provided by the iteration procedure.

The first caveat that we have in our search for dilations is that, even when $\mathbb{R}^+$ is closed under (positive) dilations, the dyadic system is not. This means that usually $K(cx,cy)$ does not even belong to $\mathscr{K}$ when $K\in\mathscr{K}$ and $c>0$. Nevertheless, Lemma ~\ref{lemma:deltahomogeneity} in \S~\ref{sec:dyadycAnalysis} gives the answer. If $K(x,y)=\varphi(\delta(x,y))$ then $K_j(x,y)=2^jK(2^jx,2^jy)=2^jK(\delta(2^jx,2^jy))=2^j\varphi(2^j\delta(x,y))$ for every $j\in \mathbb{Z}$. Hence $K_j$ depends only on $\delta$. In the next lemma we summarize the elementary properties of this mollification operator.

\begin{lemma}\label{lemma:propertiesmollificationsK}
Let $K\in\mathscr{K}^1$ with parameter $\sigma$ be given. Then $K_j(x,y)=2^jK(2^jx,2^jy)$ belongs to $\mathscr{K}^1$ with parameter $2^{-j}\sigma$. Moreover, denoting with $\varphi^{(j)}$, $\bar{\alpha}^{j}=(\alpha^j_i: i\in \mathbb{Z})$ and $\bar{\lambda}^j=(\lambda^j_i: i\in \mathbb{Z})$ the corresponding functions and sequences for each $K_j$ we have that;
\begin{enumerate}[(a)]
\item $\varphi^{(j)}(s)=2^j\varphi(2^js)$, $j\in \mathbb{Z}$, $s>0$;
\item $\alpha^j_l=\alpha_{l-j}$, $j\in \mathbb{Z}$, $l\in \mathbb{Z}$;
\item $\lambda^j_l=\lambda_{l-j}$, $j\in \mathbb{Z}$, $l\in \mathbb{Z}$.
\end{enumerate}
\end{lemma}

\begin{proof}
From the considerations above, it is clear that $K_j\in\mathscr{K}$. Now, for $j\in \mathbb{Z}$ fixed,
\begin{equation*}
\delta(x,y)^2K_j(x,y)=\delta(x,y)^2 2^j K(2^jx,2^jy)=2^{-j}\delta(2^jx,2^jy)^2K(2^jx,2^jy)
\end{equation*}
which tends to $2^{-j}\sigma$ when $\delta(x,y)\to\infty$. Property (a) is clear. Property (b) follows from (a);
\begin{align*}
\varphi^{(j)}(s)=2^j\varphi(2^js)=2^j\sum_{l\in \mathbb{Z}}\alpha_l\varphi_l(2^js)=\sum_{l\in \mathbb{Z}}\alpha_l\varphi_{l+j}(s)=\sum_{l\in \mathbb{Z}}\alpha_{l-j}\varphi_l(s).
\end{align*}
Hence $\alpha^j_l=\alpha_{l-j}$. Finally (c) follows from (b) and (4.c) in Lemma~\ref{lemma:kerneldelta1}.
\end{proof}

Corollary~\ref{coro:compositonKntimes} and Lemma~\ref{lemma:propertiesmollificationsK} show that for $K\in\mathscr{K}^1$ with parameter $\sigma$ if we iterate $K$, $2^i$-times ($i$ a positive integer) to obtain $K^{(2^i)}$  and then we mollify this kernel by a scale $2^i$, the new kernel $M^i$ belongs to $\mathscr{K}^1$ with parameter $\sigma$. Notice also that iteration and mollification commute, so that $M^i$ can be also seen as the $2^i$-th iteration of the $2^i$ mollification of $K$. Let us gather in the next statement the basic properties of $M^i$ that shall be used later, and follows from Corollary~\ref{coro:compositonKntimes} and Lemma~\ref{lemma:propertiesmollificationsK}.

\begin{lemma}
Let $K\in\mathscr{K}^1$ with parameter $\sigma$ and let $i$ be a positive integer. Then, the kernel $M^i\in\mathscr{K}^1$ with parameter $\sigma$ and $\lambda^i_j=\lambda^{2^i}_{j-i}$.
\end{lemma}

\section{The main result}\label{sec:mainresult}
We are in position to state and prove the main result of this paper. In order to avoid a notational overload in the next statement, we shall use the notation introduced in the above sections.

\begin{theorem}\label{thm:mainresult}
Let $K$ be in $\mathscr{K}^1$ with parameter $\tfrac{2}{3}t>0$. Then
\begin{enumerate}[(a)]
\item the eigenvalues of $M^i$ converge to the eigenvalues of the kernel in \eqref{eq:NucleoHaarDifusiones} when $i\to+\infty$, precisely
\begin{equation*}
\lambda^{2^i}_{j-i}\to e^{-t2^j}, \textrm{ when } i\to\infty;
\end{equation*}
\item for $1<p<\infty$ and $u_0\in L^p(\mathbb{R}^+)$, the functions $v_i(x)=\int_{\mathbb{R}^+}M^i(x,y)u_0(y) dy$ converge in the $L^p(\mathbb{R}^+)$ sense to the solution $u(x,t)$ of the problem
\begin{equation*}
(P) \left
\{\begin{array}{ll}
\frac{\partial u}{\partial t}=D^{1} u,\, & x\in\mathbb{R}^{+}, t>0;\\

u(x,0)=u_0(x),\,  & x\in \mathbb{R}^+.
\end{array}
\right.
\end{equation*}
for the precise value of $t$ for which the initial kernel $K$ is $1$-stable with parameter $\tfrac{2}{3}t$.
\end{enumerate}
\end{theorem}
\begin{proof}[Proof of (a)]
Since $K\in\mathscr{K}^1$ with parameter $\tfrac{2}{3}t>0$, which means that $k_m4^m\to\tfrac{2}{3}t$ as $m$ tends to infinity we have both that $k_m2^m\to 0$ when $m\to\infty$ and that $\sum_{l<m}k_l2^{l-1}<1$ for every positive integer $m$. Since, on the other hand $\sum_{l\in \mathbb{Z}}k_l2^{l-1}=1$, we have for $j\in \mathbb{Z}$ fixed and $i$ a large nonnegative integer that
\begin{equation*}
0<\sum_{l<i-j}k_l2^{l-1}- \frac{k_{i-j}2^{i-j}}{2}<1.
\end{equation*}
Hence, from Lemma~\ref{lemma:propertiesmollificationsK} and Lemma~\ref{lemma:kerneldelta1}, the $j$-th scale eigenvalues of the operator induced by the kernel $M^i$ ar given by
\begin{align*}
\lambda^{2^i}_{j-i}&=\left[\frac{1}{2}\left(\sum_{l<i-j}k_l2^l-k_{i-j}2^{i-j}\right)\right]^{2^i}\\
&=\left[\sum_{l<i-j}k_l2^{l-1}-k_{i-j}\frac{2^{i-j}}{2}\right]^{2^i}\\
&=\left[1-\left(\sum_{l\geq i-j}k_l2^{l-1}+\frac{k_{i-j}4^{i-j}}{2}\frac{2^j}{2^i}\right)\right]^{2^i}\\
&= \left[1-\gamma(i,j)\frac{2^j}{2^i}\right]^{2^i},
\end{align*}
with $\gamma(i,j)=2^{i-j}\sum_{l\geq i-j}k_l2^{l-1}+\frac{k_{i-j}4^{i-j}}{2}$. Notice that
\begin{equation*}
\gamma(i,j)=2^{i-j}\sum_{l\geq i-j}2^{-l-1}(k_l4^l)+\frac{k_{i-j}4^{i-j}}{2}=\sum_{m=0}^{\infty}2^{-m-1}(k_{i+m-j}4^{i+m-j})+\frac{k_{i-j}4^{i-j}}{2},
\end{equation*}
which tends to $t>0$ when $i\to\infty$. With these remarks we can write
\begin{equation*}
\lambda^{2^i}_{j-i}=\left(\left[1-\frac{\gamma(i,j)2^j}{2^i}\right]^{\tfrac{2^i}{\gamma(i,j)2^j}}\right)^{\gamma(i,j)2^j}
\end{equation*}
which tends to $e^{-t2^j}$ when $i$ tends to infinity.

\textit{Proof of (b).}
The function $v_i(x)-u(x,t)$ can be seen as the difference of two operators $T_i$ and $T^t_{\infty}$ acting on the initial condition,
\begin{equation*}
v_i(x)=T_iu_0(x)=\int_{y\in \mathbb{R}^+}M^i(x,y)u_0(y) dy
\end{equation*}
and
\begin{equation*}
u(x,t)=T^t_{\infty}u_0(x)=\int_{y\in \mathbb{R}^+}K(x,y;t)u_0(y)dy.
\end{equation*}
Since the eigenvalues of $T_i-T^t_\infty$ are given by $\lambda^{2^i}_{j(h)-i}-e^{-t2^{j(h)}}$, for each $h\in\mathscr{H}$, from Theorem~\ref{thm:op.Lp.haar} in Section~\ref{sec:spectralanalysis} we have
\begin{equation*}
\norm{v_i-u(\cdot,t)}_{L_p(\mathbb{R}^+)}\leq C_1\biggl\|\biggl(\sum_{h\in\mathscr{H}}\abs{\lambda^{2^i}_{j(h)-i}-e^{-t2^{j(h)}}}^2\abs{\proin{u_0}{h}}^2
\abs{I(h)}^{-1}\chi_{I(h)}(\cdot)\biggr)^{1/2}\biggr\|_{L_p(\mathbb{R}^+)}.
\end{equation*}
From (5.g) and (5.h) in Lemma~\ref{lemma:kerneldelta1} we have that the sequence $\lambda^{2^i}_{j(h)-i}$ is uniformly bounded. On the other hand, since $\norm{\bigl(\sum_{h\in\mathscr{H}}\abs{\proin{u_0}{h}}^2
{\abs{I(h)}}^{-1}\chi_{I(h)}(\cdot)\bigr)^{1/2}}_{L_p(\mathbb{R}^+)}\leq C_2\norm{u_0}_{L^p(\mathbb{R}^+)}<\infty$, we can take the limit for $i\to+\infty$ inside the $L^p$-norm and the series in order to get that $\norm{v_i-u(\cdot,t)}_{L_p(\mathbb{R}^+)}\to 0$ when $i\to+\infty$.
\end{proof}



\def\cprime{$'$}
\providecommand{\bysame}{\leavevmode\hbox to3em{\hrulefill}\thinspace}
\providecommand{\MR}{\relax\ifhmode\unskip\space\fi MR }
\providecommand{\MRhref}[2]{%
  \href{http://www.ams.org/mathscinet-getitem?mr=#1}{#2}
}
\providecommand{\href}[2]{#2}


\bigskip

\bigskip
\noindent{\footnotesize
\textsc{Instituto de Matem\'atica Aplicada del Litoral, UNL, CONICET}

\smallskip
\noindent\textmd{CCT CONICET Santa Fe, Predio ``Dr. Alberto Cassano'', Colectora Ruta Nac.~168 km 0, Paraje El Pozo, S3007ABA Santa Fe, Argentina.}
}
\bigskip

\end{document}